\documentclass[12pt]{amsart}
\usepackage{amscd}
\usepackage{amssymb,latexsym}
\usepackage{bm}

\usepackage[pdftex,colorlinks=true]{hyperref}
\usepackage[T1]{fontenc}
\usepackage{charter}\usepackage{eulervm}
\usepackage{amsaddr}

\theoremstyle{plain}
\newtheorem{thm}{Theorem}[section]
\newtheorem*{main}{Main Theorem}
\newtheorem*{thmclas}{Product property}
\newtheorem*{add}{Additivity property}
\newtheorem*{otopy_in}{Otopy invariance}
\newtheorem*{solution}{Solution property}
\newtheorem*{normal}{Normalization property}
\newtheorem{lem}[thm]{Lemma}
\newtheorem{prop}[thm]{Proposition}
\newtheorem{cor}[thm]{Corollary}

\theoremstyle{definition}

\theoremstyle{remark}
\newtheorem{rem}[thm]{Remark}

\numberwithin{equation}{section}

\newcommand{\wt}{\widetilde}

\newcommand{\R}{\mathbb{R}}

\newcommand{\Z}{\mathbb{Z}}

\newcommand{\restrictionmap}[2]{{#1}\mathpunct\restriction\hbox{}_{#2}}
\providecommand{\abs}[1]{\left\lvert#1\right\rvert}

\DeclareMathOperator{\Loc}{Loc}
\DeclareMathOperator{\id}{Id}
\DeclareMathOperator{\cl}{cl}

\DeclareMathOperator{\I}{I}
\DeclareMathOperator{\Iso}{Iso}

\DeclareMathOperator{\rank}{rank}

\title[Degree product formula]
{Degree product formula in the case\\ of a finite group action} 

\author[P. Bart{\l}omiejczyk, B. Kamedulski and P. Nowak-Przygodzki]
{Piotr Bart{\l}omiejczyk, Bartosz Kamedulski 
and Piotr Nowak-Przygodzki}
\author{}
\address{Faculty of Applied Physics and Mathematics,
Gda{\'n}sk University of Technology,
Gabriela Narutowicza 11/12,
80-233 Gda{\'{n}}sk, Poland}
\email{piobartl@pg.edu.pl}  
\email{bartoszkamedulski@gmail.com}
\email{piotrnp@wp.pl}

\date{\today}

\begin{document}

\begin{abstract}
Let $V, W$ be finite-dimensional orthogonal representations of 
a finite group $G$. The equivariant degree with values
in the Burnside ring of $G$ has been studied extensively by many authors. 
We present a short proof of the degree product formula 
for local equivariant maps on $V$ and $W$.
\end{abstract}

\maketitle


\section*{Introduction}\label{sec:intro} 
One of the basic properties of the topological degree is 
the \emph{product property}. Recall that a continuous map 
from an open subset of $\R^n$ into $\R^n$ is called \emph{local} 
if its set of zeros is compact.
For such maps the classical Brouwer degree $\deg$ is well-defined
and the product property holds. Namely,
\begin{thmclas}[{\cite[Prop. 8.7]{Br}}]
Let $f\colon D_f\subset\R^m\to\R^m$ and $f'\colon D_{f'}\subset\R^n\to\R^n$
be local maps. Then $f\times f'\colon D_f\times D_{f'}\to\R^{m+n}$ 
is also a local map and
\[
\deg(f\times f')=\deg f\cdot\deg f'.
\]
\end{thmclas}

Our main goal is to present a short proof of an equivariant version
of the product formula for equivariant local maps
in the case of a~finite group action.
In that case the formula has an analogous form
\[
\deg_G(f\times f')=\deg_G f\cdot\deg_G f',
\]
but since the equivariant degree $\deg_G$ has its values in 
the \emph{Burnside ring} of a finite group $G$, the multiplication
on the right side of the formula takes place in this Burnside ring.
It is worth pointing out that in \cite{GKW} the authors proved
the equivariant product formula in much more general setting
i.e.\ in the case of a compact Lie group action. Unfortunately, 
this proof seems to be rather sketchy in some parts. We hope that
our proof has the advantage of being straightforward and complete
and can be seen as the first step towards proving the general case.

The paper is organized as follows. 
Section \ref{sec:basic} contains preliminaries. 
In Section \ref{sec:degree} we recall 
the concept of the equivariant degree $\deg_G$. 
Our main result is stated in Section \ref{sec:main}. 
In Section \ref{sec:poly} we introduce standard and polystandard maps
and study their properties needed in the next section. 
Finally, Section \ref{sec:proof} contains the proof of our main result.


\section{Basic definitions} 
\label{sec:basic}

\subsection{Local maps}
The notation $A\Subset B$ means 
that $A$ is a compact subset of $B$.
For a topological space $X$,
we denote by $\tau(X)$ the topology on $X$.
For any topological spaces $X$ and $Y$, 
let $\mathcal{M}(X,Y)$ be the set
of all continuous maps $f\colon D_f\to Y$ such that 
$D_f$ is an open subset of $X$. 
Let $\mathcal{R}$ be
a family of subsets of $Y$. We define 
\[
\Loc(X,Y,\mathcal{R}):=\{\,f\in\mathcal{M}(X,Y)\mid 
f^{-1}(R)\Subset D_f \text{ for all $R\in\mathcal{R}$}\,\}.
\]
We introduce a topology in $\Loc(X,Y,\mathcal{R})$
generated by the subbasis consisting of all sets of the form
\begin{itemize}
	\item $H(C,U):=\{\,f\in\Loc(X,Y,\mathcal{R})\mid 
	      C\subset D_f,\, f(C)\subset U\,\}$
	      for $C\Subset X$ and $U\in\tau(Y)$,
	\item $M(V,R):=\{\,f\in\Loc(X,Y,\mathcal{R})\mid 
	      f^{-1}(R)\subset V\,\}$ for $V\in\tau(X)$ and 
	      $R\in\mathcal{R}$.
\end{itemize}
Elements of $\Loc(X,Y,\mathcal{R})$ 
are called \emph{local maps}.
The natural base point of $\Loc(X,Y,\mathcal{R})$ 
is the empty map.
Let $\sqcup$ denote the union of two disjoint local maps.
Moreover, in the case when $\mathcal{R}=\{\{y\}\}$
we will write $\Loc(X,Y,y)$ omitting double curly brackets.
For more details we refer the reader to \cite{BP1}.

\subsection{Equivariant maps}
Assume that $V$ is a real finite dimensional orthogonal
representation of a finite group $G$.
Let $X$ be an arbitrary $G$-space.
We say that $f\colon X\to V$ is \emph{equivariant},
if $f(gx)=gf(x)$ for all $x\in X$ and $g\in G$.
We will denote by 
$
\mathcal{C}_G(X)
$
the space $\{f\in\Loc(X,V,0)\mid\text{$f$ is equivariant}\}$
with the induced topology.
Assume that $\Omega$ is an open invariant subset of $V$.
Elements of $\mathcal{C}_G(\Omega)$ are called
\emph{equivariant local maps}.

\subsection{Otopies}
Let $I=[0,1]$. 
We assume that the action
of $G$ on $I$ is trivial.
Any element of 
$\mathcal{C}_G(I\times\Omega)$
is called an \emph{otopy}.
Each otopy corresponds to a path in $\mathcal{C}_G(\Omega)$ 
and vice versa.
Given an otopy 
$h\colon\Lambda\subset I\times\Omega\to V$ 
we can define for each $t\in I$:
\begin{itemize}
	\item sets $\Lambda_t=\{x\in\Omega\mid(t,x)\in\Lambda\}$,
	\item maps $h_t\colon\Lambda_t\to V$ with $h_t(x)=h(t,x)$.
\end{itemize}
In this situation we say that 
$h_0$ and $h_1$ are \emph{otopic}.
Otopy gives an equivalence relation 
on $\mathcal{C}_G(\Omega)$.
The set of otopy classes will be denoted by 
$\mathcal{C}_G[\Omega]$.
\begin{rem}\label{rem:loc}
Observe that if $f\in\mathcal{C}_G(\Omega)$
and $U$ is an open invariant subset of $D_f$
such that $f^{-1}(0)\subset U$, then
$f$ and $\restrictionmap{f}{U}$ are otopic. In particular, if
$f^{-1}(0)=\emptyset$ then $f$ is otopic to the empty map.
\end{rem}

\subsection{\texorpdfstring{$G$}{G}-actions}
If $H$ is a subgroup of $G$ then
\begin{itemize}
\item $(H)$ stands for the conjugacy class of $H$,
\item $NH$ is the normalizer of $H$ in $G$,
\item $WH$ is the Weyl group of $H$ i.e. $WH=NH/H$.
\end{itemize}
Recall that $G_x=\{g\in G\mid gx=x\}$.
We define the following subsets of~$V$:
\begin{align*}
V^H &=\{x\in V\mid H\subset G_x\},\\
\Omega_H &=\{x\in\Omega\mid H=G_x\}.
\end{align*}%
The set $\Iso(\Omega):=\{(H)\mid
\text{$H$ is a closed subgroup of $G$ and $\Omega_{H}\neq\emptyset$}\}$
is partially ordered. Namely, $(H)\le(K)$ if $H$ is conjugate to a subgroup of $K$.
	
We will make use of
the following well-known facts:
\begin{itemize}
	\item $V^H$ is a linear subspace of $V$ and an orthogonal
	      representation of $WH$,
	\item $\Omega_H$ is open in $V^H$,	      
	\item the action of $WH$ on $\Omega_H$ is free.	
\end{itemize}

\subsection{Burnside ring} 
\label{sec:burnside}
Let $\mathcal{A}^+(G)$ be the set 
of isomorphism classes of finite $G$-sets. 
While disjoint union of finite $G$-sets 
induces addition on $\mathcal{A}^+(G)$, 
cartesian product with diagonal action 
induces multiplication, i.e.\
\[
[X]+[Y]=[X\sqcup Y],\quad[X]\cdot[Y]=[X\times Y],
\] 
where $[X]$, $[Y]$ are isomorphism classes of finite $G$-sets.
The resulting structure is a commutative semi-ring with identity.

Since every finite $G$-set is a disjoint union of its orbits, 
each element of the semi-ring can be presented uniquely 
as $\sum d_{(H)}[^{G}\!/\!_H],$ 
where each $d_{(H)}$ is a non-negative integer and $[^{G}\!/\!_H]$ 
is the isomorphism class of $^{G}\!/\!_H$, 
which depends only on the conjugacy class of $H$. 
The problem of decomposing ${^{G}\!/\!_H}\times{^{G}\!/\!_K}$ 
into orbits makes multiplication in $\mathcal{A}^+(G)$ non-trivial.

The Grothendieck ring constructed from $\mathcal{A}^+(G)$ 
is denoted by $\mathcal{A}(G)$ and 
called the \textit{Burnside ring} of $G$. 
Additively, it is a free abelian group generated by 
isomorphism classes $[^{G}\!/\!_H]$ of $^{G}\!/\!_H$. 
$\mathcal{A}(G)$ is a commutative ring 
with the unit $[^{G}\!/\!_G].$

\subsection{Local cross sections of a vector bundle} 
\label{sec:cross}
All manifolds considered are without boundary.
Assume $p\colon E\to M$ is a smooth (i.e., $C^1$) vector bundle. 
We will identify $M$ with the zero section of $E$.
A \emph{local cross section} of a bundle $p\colon E\to M$
is a continuous map $s\colon U\to E$, where $U$ is open in $M$,
$s^{-1}(M)$ is compact and $p\circ s=\id_U$. 
Let $\Gamma(M,E)$ denote 
the set of all local cross sections of $E$ over $M$.

Assume that $\rank E=\dim M$ and $E$ is orientable as a manifold.
Let us denote by $\I(s)$ the oriented intersection number
of a local cross section $s$ with the zero section
(see for instance \cite{GP}), which is an integer.


\section{Degree \texorpdfstring{$\deg_G$}{deg G}}
\label{sec:degree}

In papers \cite{B1,BGI,GKW} the authors introduce
the equivariant degree $\deg_G: \mathcal{C}_G(V) \to \mathcal{A}(G)$
for the action of a compact Lie group $G$
and prove that the degree has the following expected properties.

\begin{add}
If $f,f'\in\mathcal{C}_G(V)$ and $D_f\cap D_{f'}=\emptyset$ then
\[
\deg_G (f\sqcup f')=\deg_G f+\deg_G f'.
\]
\end{add}

\begin{otopy_in}
Let $f,f'\in\mathcal{C}_G(V)$. If $f$ and $f'$ are otopic then 
\[
\deg_G f = \deg_G f'.
\]
\end{otopy_in}

\begin{solution}
If $\deg_G f \neq 0$ then $f(x)=0$ for some $x \in D_f$.
\end{solution}

In order to formulate the next property, it is necessary 
to introduce some notation and make some assumptions.
Let $B(p,r)$ denote the open $r$-ball in $V$ around $p$.
Assume that $G$ is finite, $x\in V$ and 
$f\colon\cup_{y\in Gx}B(y,\epsilon)\to V$,
where $\epsilon<\frac12\min\{\abs{a-b}\mid a,b\in Gx, a\neq b\}$.

\begin{normal}
If $f(y+v)=v$ for $y\in Gx$ and $\abs{v}<\epsilon$,
then $f\in\mathcal{C}_G(V)$, $f^{-1}(0)=Gx$ and $\deg_G f=[G/G_x]$.
\end{normal}

\begin{rem}
The equivariant degree, as an element of the Burnside ring, 
consists of multiple coefficients. 
Examination of these allows not only to find orbits of zeros, 
but also to analyze their orbit types.
\end{rem}

Recall here that the main goal of this paper is to show
that the degree $\deg_G$ has the \emph{product property} as well.


\section{Main result} 
\label{sec:main}
Assume that
\begin{itemize}
	\item $G$ is a finite group,
	\item $V$ and $W$ are real finite dimensional orthogonal
  representations of $G$.
\end{itemize}
Recall that $\mathcal{C}_G(V)$ denotes 
the space of equivariant local maps in $V$.

\begin{main}
If $f\in\mathcal{C}_G(V)$ and $f'\in\mathcal{C}_G(W)$, 
then $f\times f'\in\mathcal{C}_G(V\oplus W)$ and
\[
\deg_G(f\times f')=\deg_G f\cdot\deg_G f',
\]
where ``$\cdot$'' denotes the multiplication 
in the Burnside ring $\mathcal{A}(G)$.
\end{main}


\section{Standard and polystandard maps}
\label{sec:poly}
In this section we introduce standard and polystandard maps
and study their basic properties. These maps will play the crucial role 
in the proof of Main Theorem in the next section. Let us start
with recalling the definition of an $\epsilon$-normal neighbourhood. 
Assume that:
\begin{itemize}
    \item $Y$ is a linear subspace of $\R^n$,
    \item $U$ is an open subset of $Y$.
    \end{itemize}
Let $Y^\bot$ denote the orthogonal complement of $Y$ in $\R^n$.
For $\epsilon>0$ let us denote by $U^\epsilon$ the set
$U^\epsilon=\{x+v\mid x\in U, v\in Y^\bot, \abs{v}<\epsilon\}$.
Any such set will be called an 
$\epsilon$-\emph{normal neighbourhood} of $U$.

Now we are ready to introduce two important classes of maps: 
standard and polystandard.
A map $f\in\mathcal{C}_G(V)$ is called \emph{standard} if
\begin{itemize}
    \item $f^{-1}(0)=Gx_0$ for some $x_0 \in D_f$,
    \item there is an open subset $U$ of $V_H$, where $H=G_{x_0}$, 
		and $\epsilon >0$ such that:
    \begin{itemize}
        \item $f^{-1}(0)\cap U=\{x_0\}$,
        \item $U^\epsilon\subset D_f$,
        \item $f(x+v)=f(x)+v$ for all 
				$x \in U, v \in (V^H)^\bot, |v|<\epsilon$.
    \end{itemize}
\end{itemize}
In such situation we also say that $f$ is standard with respect to $x_0$, 
$U$ and $\epsilon$.
A map $f \in \mathcal{C}_G(V)$ is called $m$-\emph{standard} 
if there are standard maps $f_i$ ($i=1,\dotsc,m$) 
with disjoint domains such that:
\[
f^{-1}(0)\subset\sqcup_{i=1}^m D_{f_i}\subset D_f.
\]
If a map is $m$-standard for some $m$, 
we call it \emph{polystandard}.
A finite disjoint union of standard maps is called 
\emph{strictly polystandard}. By Remark \ref{rem:loc},
any polystandard map is otopic to a strictly polystandard one.

In what follows, we will need the notation that
relates the classical topological and equivariant degrees.
Let $f$ be standard with respect to $x$, $U$ and $\epsilon$
and let $\alpha=Gx$. Consider a map $f_x\colon U\subset V_H\to V^H$ 
given by $f_x=\restrictionmap{f}{U}$. Let $d_x=\deg(f_x,U)$.

\begin{prop}\label{prop:deg}
For each $g\in G$ the equality $d_x=d_{gx}$ holds.
\end{prop}

\begin{proof}
Let $g\in G$ and $K=G_{gx}$. Then $V^K=gV^H$.
Consider a map $f_{gx}\colon gU\subset V_K\to V^K$ satisfying 
$f_{gx}(y)=gf_x(g^{-1}y)$. Since $g\colon V^H \to V^K$ 
is an isomorphism of linear spaces, 
\[
d_{gx}=\deg(f_{gx}, gU)=\deg(f_x, U)=d_x.\qedhere
\]
\end{proof}
Define $d_\alpha=d_y$, where $y$ is any element of $\alpha$.
Proposition \ref{prop:deg} guarantees that the integer $d_\alpha$
is well-defined.
The main advantage of standard and polystandard maps is that
we can immediately compute their equivariant degree $\deg_G$
if we know the value of $d_\alpha$.
Namely, let $f$ be a standard map and $\alpha=Gx=f^{-1}(0)$. Then
\[
\deg_G f=d_\alpha [G/G_{x}]=d_\alpha[\alpha].
\]
More generally, let $f$ be a $m$-standard map, and let 
$\{\alpha_i\}_{i=1}^m$ denote the set of orbits of zeros of $f$. 
Then
\[
\deg_G f=\sum_{i=1}^m d_{\alpha_i}[\alpha_i].
\]


\section{Proof of Main Theorem}
\label{sec:proof}
 
To prove  Main Theorem we will need two lemmas.

\begin{lem}\label{lem:poly}
Let $f, f'$ be standard maps and 
$\alpha=f^{-1}(0)$, $\beta=(f')^{-1}(0).$ 
Then $f \times f'$ is polystandard and for each orbit 
$\gamma \subset \alpha \times \beta$ we have:
\[
d_\gamma=d_\alpha \cdot d_\beta.
\]
Moreover,
\[
\deg_G(f\times f')=\deg_G f\cdot\deg_G f'.
\]
\end{lem}

\begin{proof}
First we show that $f\times f'$ is polystandard.
Assume that $f$ is standard with respect to $x_0$, $U$ and $\epsilon$
and $f'$ is standard with respect to $x'_0$, $U'$ and $\epsilon'$.
Let $H=G_{x_0}$ and $K=G_{x'_0}$. 
Note that $(f \times f')^{-1}(0)=Gx_0 \times Gx'_0$ 
is a finite union of orbits and $G_{(x_0,x'_0)}=H\cap K.$ 
Since $U^{\epsilon} \times {U'}^{\epsilon'}$ is open in $V\oplus W$ and 
\[
(x_0,x'_0)\in V_H\times W_K\subset (V \oplus W)_{H\cap K},
\] 
there exists an open subset 
$U''\subset (V \oplus W)_{H\cap K}$ and $\epsilon''>0$ 
such that $(f \times f')^{-1}(0) \cap U'' = \{(x_0, x'_0)\}$ and 
${U''}^{\epsilon''} \subset U^\epsilon \times {U'}^{\epsilon'}$. 

Now let us check that 
$(f\times f')(x''+w'')=(f\times f')(x'')+w''$
for $x''\in U''$, $w''\in ((V \oplus W)^{H\cap K})^\bot$, $\abs{w''}<\epsilon''$.
Note that since ${U''}\subset U^\epsilon\times {U'}^{\epsilon'}$,
$x''$ has the unique representation as $(x,x')+(v,v')$,
where $(x,x')\in U\times U'$ and $(v,v')\in (V^H)^\bot\oplus (W^K)^\bot$.
Moreover, since 
\[
((V \oplus W)^{H\cap K})^\bot\subset(V^H)^\bot\oplus(W^K)^\bot,
\]
$w''$ can be uniquely written as $(w,w')$ with 
$w \in (V^H)^\bot$ and $w' \in (W^K)^\bot$. Hence
\begin{multline*}
(f\times f')(x''+w'')=(f \times f')(x+v+w,x'+v'+w')\\
=\big(f(x+v+w),f'(x'+v'+w')\big)
=\big(f(x)+v+w, f'(x')+v'+w'\big)\\
=\big(f(x+v)+w, f'(x'+v')+w'\big)
=(f \times f')(x'')+w'',
\end{multline*}
which proves that $f\times f'$ is polystandard.

Next we show the formula $d_\gamma=d_\alpha \cdot d_\beta$.
Let $(x_0, x'_0)\in\gamma\subset\alpha\times\beta$. Observe that
\begin{align*}
d_\alpha&=d_{x_0}=\deg(f_{x_0},U)\stackrel{1}{=}\deg(f,U^\epsilon), \\
d_\beta&=d_{x'_0}=\deg(f'_{x'_0}, U')\stackrel{1}{=}\deg(f',{U'}^{\epsilon'}).
\end{align*}
Thus we get
\begin{multline*}
    d_\gamma=d_{(x_0,x'_0)}=\deg((f\times f')_{(x_0,x'_0)}, U'')
		\stackrel{1}{=}\deg(f \times f', {U''}^{\epsilon''})\\
    \stackrel{2}{=}\deg(f \times f', U^{\epsilon} \times {U'}^{\epsilon'}) 
		\stackrel{1}{=}\deg(f,U^\epsilon)\cdot\deg(f',{U'}^{\epsilon'})
		=d_\alpha\cdot d_\beta.
\end{multline*}
In the above we used two properties of the classical topological degree:
the  product formula ($1$) and the localization of zeros ($2$).

Finally, we prove the product formula for standard maps.
As we have shown $f \times f'$ is $m$-standard for some $m$.
Decompose $\alpha\times\beta$ into the disjoint union of orbits 
$\bigsqcup_{i=1}^m\gamma_i$. We have 
$d_{\gamma_i}=d_\alpha \cdot d_\beta$ for each $i=1,2,\dotsc,m$.
Thus we get
\begin{multline*}
\deg_G(f\times f')=\sum_{i=1}^m d_{\gamma_i} [\gamma_i]=\sum_{i=1}^m d_\alpha d_\beta [\gamma_i]
=d_\alpha d_\beta\sum_{i=1}^m  [\gamma_i]\\
=d_\alpha d_\beta [\alpha \times \beta]
=d_\alpha [\alpha] \cdot d_\beta[\beta]
=\deg_G f \cdot \deg_G f',
\end{multline*}
which establishes the desired formula.
\end{proof}

We precede the next lemma by recalling the following notation.
Orbit types are indexed (according to the partial order) 
by natural numbers $i=1,2,\dotsc,m$. In particular, $H_1=G$.
Write $M_i=V_{H_i}/WH_i$ and $E_i=\left(V_{H_i}\times V^{H_i}\right)/WH_i$.
Recall that $p_i\colon E_i\to M_i$ is a vector bundle
such that $\rank E_i=\dim M_i$ and $E_i$ is orientable as a manifold.
Moreover, the bundle $E_i\to M_i$ is naturally isomorphic
to the tangent bundle $TM_i\to M_i$. We will denote by $\Gamma(M_i,E_i)$
the set of local cross sections of the bundle $p_i$ and by 
$\mathcal{B}_i\colon\mathcal{C}_{WH_i}\left(V_{H_i}\right)\to\Gamma(M_i,E_i)$
the function defined by the formula $\mathcal{B}_i(f)([x])=[(x,f(x)]$, 
where $x\in V_{H_i}$. Since $WH_i$ acts freely on $V_{H_i}$,
the function $\mathcal{B}_i$ is a bijection (see \cite[Thm 4.1]{B1}).
Let $\{M_{ij}\}_j$ denote 
the set of connected components of the manifold $M_i$ 
and $n(i)$ denote the number of these components,
which is finite or countable.
The formulation of the next lemma requires also the function
\[
\Phi\colon\mathcal{C}_{G}(V)\to
\prod_{i=1}^{m}\bigg(\sum_{j=1}^{n(i)}\mathbb{Z}\bigg)
\] 
defined in \cite[Thm 5.2]{B1}.
The function $\Phi$ has the following properties
\begin{itemize}
	\item $\Phi(\mathcal{C}_G(V))=\begin{cases}
  \prod_{i=1}^{m}\left(\sum_{j=1}^{n(i)}\mathbb{Z}\right) & \text{if $\dim V^G>0$},\\
  \{0,1\}\times\prod_{i=2}^{m}\left(\sum_{j=1}^{n(i)}\mathbb{Z}\right)& \text{if $\dim V^G=0$},
  \end{cases}$
	\item the function induced by $\Phi$ on $\mathcal{C}_G[V]$, which will be denoted
	by the same letter, is an injection.
\end{itemize}

\begin{lem}\label{lem:realization}
For any system $\{c_{ij}\}\in\Phi(\mathcal{C}_G(V))$ 
there is a strictly polystandard map $f$ such that $\Phi(f)=\{c_{ij}\}$.
\end{lem}

\begin{proof}
We need to consider two cases.

\noindent\textsc{Case 1:} $\dim V^G>0$. Under that assumption we have
$\dim M_i>0$ for each $i=1,2,\dotsc,m$. 
Fix $\{c_{ij}\}\in\Phi(\mathcal{C}_G(V))$, where $c_{ij}\in\Z$. 
On the component $M_{ij}$ choose $\abs{c_{ij}}$ points
together with their disjoint disc neighbourhoods. 
Let us denoteby $P_{ij}$ the set of these points 
and by $F_{ij}$ the union of their neighbourhoods. 
By Corollary \ref{cor:index} from Appendix~\ref{sec:dodatekA},
there is a local cross section 
$s_{ij}\colon F_{ij}\subset M_{ij}\to E_i$ such that
\[
	s_{ij}^{-1}(M_{ij})=P_{ij}\quad\text{and}\quad I(s_{ij})= c_{ij}.
\]
Next we define a local cross section 
$s_{i}\colon F_i\subset M_{i}\to E_i$
as a disjoint union $s_i=\sqcup_{j=1}^{n(i)}s_{ij}$. 
Note that the set $\cup_{i,j}P_{ij}$ is finite,
because only a finite number of $c_{ij}$ are nonzero.
Set $f_i=\mathcal{B}_i^{-1}(s_i)$. 
By the definition of $\mathcal{B}_i$, 
$f_i\in\mathcal{C}_{WH_i}\left(V_{H_i}\right)$
and by the construction of $s_i$, 
the domain $D_{f_i}$ is the union of disjoint discs
around all points in $f_i^{-1}(0)$. Consequently,
the set $D:=\cup_{i=1}^{m}D_{f_i}$ is a finite
disjoint union of discs $R_k$ such that every disc
contains one zero of a given $f_i$. Thus $D=\sqcup_k R_k$.
Observe that there is $\epsilon>0$ such that the sets
$R_k^\epsilon$ are pairwise disjoint. 
Since any point of $\sqcup_k R_k^\epsilon$
can be uniquely represented in the form $gx+gv$,
where $x\in D_{f_i}$, $v\in(V^{H_i})^\bot$, 
$\abs{v}<\epsilon$, let us define 
$f\colon D_f:=\sqcup_{k}R_k^\epsilon\to V$~by 
\[
f(gx+gv)=gf_i(x)+gv
\]
for all $g\in G$, $x\in D_{f_i}$, $v\in(V^{H_i})^\bot$, 
$\abs{v}<\epsilon$.
Our construction guarantees that
\begin{itemize}
	\item $f\in\mathcal{C}_G(V)$,
	\item $f$ is strictly polystandard,
	\item $\Phi(f)=\{c_{ij}\}$.
\end{itemize}

\noindent\textsc{Case 2:} $\dim V^G=0$. In that situation, $M_1=\{0\}$ and
$\dim M_i>0$ for $i>1$. Analogously as in the previous case, we define
$f\colon D_f\to V$, but now in the construction of $f$ we take into
account $M_i$ only for $i>1$. By choosing the disc neighbourhoods
small enough, we can guarantee that $0\not\in\cl(D_f)$.
Hence there is $\delta>0$ such that 
$B(0,\delta)\cap D_f=\emptyset$, where $B(0,\delta)$ denotes 
the open $\delta$-ball in $V$ around the origin. Set
\[
\wt{f}=\begin{cases}
  f & \text{if $c_{11}=0$},\\
  f\sqcup\restrictionmap{\id}{B(0,\delta)}& \text{if $c_{11}=1$}.
  \end{cases}
\]
It is easy to see that $\Phi\big(\,\wt{f}\,\big)=\{c_{ij}\}$.
\end{proof}

\begin{cor}\label{cor:otop}
In each otopy class in $\mathcal{C}_G(V)$ 
there is a strictly polystandard map.
\end{cor}

\begin{proof}
Recall that $\Phi\colon\mathcal{C}_G[V]\to\prod\big(\sum \mathbb{Z}\big)$ 
is an injection. 
Let $[f] \in \mathcal{C}_G[V].$ By Lemma \ref{lem:realization},
there is a strictly polystandard map $f' \in \mathcal{C}_G(V)$ 
such that $\Phi([f'])=\Phi([f]).$ 
From the injectivity of $\Phi$, $f' \in [f]$.
\end{proof}

It occurs that Main Theorem is now a consequence 
of Lemma \ref{lem:poly} and Corollary \ref{cor:otop}.
\begin{proof}[Proof of Main Theorem]
The fact that $f\times f'\in\mathcal{C}_G(V\oplus W)$ is obvious.
By Corollary \ref{cor:otop}, $f$ and $f'$ are otopic 
to strictly polystandard maps $\sqcup_k f_k$ and $\sqcup_l f'_l$ 
respectively, where $f_k$ and $f'_l$ are standard. 
In consequence, $f\times f'$ is otopic to 
$\sqcup_k f_k\times\sqcup_l f'_l=\sqcup_{k,l}(f_k\times f'_l)$. Hence
\begin{multline*}
\deg_G(f\times f')\stackrel{1}{=}\deg_G\sqcup_{k,l}(f_k\times f'_l)
\stackrel{2}{=}\sum_{k,l}\deg_G(f_k\times f'_l)\\
\stackrel{3}{=}\sum_{k,l}\deg_G f_k\cdot\deg_G f'_l=
\Big(\sum_{k}\deg_G f_k\Big)\cdot\Big(\sum_{l}\deg_G f'_l\Big)\\
\stackrel{2}{=}
\Big(\deg_G\sqcup_{k} f_k\Big)\cdot\Big(\deg_G\sqcup_{l} f'_l\Big)
\stackrel{1}{=}\deg_G f\cdot\deg_G f'
\end{multline*}
from the otopy invariance property ($1$),
the additivity property ($2$)
and Lemma \ref{lem:poly} ($3$). This completes the proof.
\end{proof}

\begin{rem}
Apart from the equivariant degree $\deg_G$, 
the equivariant gradient degree $\deg_G^\nabla$
with values in the Euler-tom Dieck ring $\mathcal{U}(G)$ 
is also considered, studied and applied
(see for example \cite{BGI,BP2,GR,R}). 
In many situations $\deg_G^\nabla$ 
gives more information than $\deg_G$.
However, for a finite group the Burnside ring $\mathcal{A}(G)$ 
and the Euler-tom Dieck ring $\mathcal{U}(G)$ are identical. 
Moreover, $\deg_G=\deg_G^\nabla$ if we restrict ourselves 
to equivariant gradient local maps.
In consequence, in the case of a finite group action
the product formula holds also for $\deg_G^\nabla$.
\end{rem}

\appendix
\section{} 
\label{sec:dodatekA}

Assume that $p\colon E\to M$ is 
a vector bundle over a manifold $M$ 
such that $\dim M>0$, $\rank E = \dim M$ 
and $E$ is orientable as a manifold.

\begin{lem}
For any $q \in M$, any disc neighbourhood $D$ 
of $q$ and any $\alpha\in\{-1,1\}$ 
there is a local cross section $s\colon D\subset M \to E$ 
such that $s^{-1}(M)=\{q\}$ and $I(s)=\alpha$.
\end{lem}

\begin{proof}
Let $B=\{x \in\R^n\mid\abs{x}<1\}$ and $TB=B \times \mathbb{R}^n$. 
Consider a local cross section $s_A\colon B\to TB$ given by 
$s_A(x)=(x, Ax)$, where $A$ is linear and $\det A=\alpha$. 
Observe that $I(s_A)=\alpha$. 

Let us note that we can identify $TD$ with $E\mid_D$. 
Take a diffeomorphism $\varphi: B \to D$ such that $\varphi(0)=q$. 
It induces a tangent map $T\varphi\colon TB \to TD$. 
Finally, define a local cross section $s\colon D\to E$ by formula 
$s(x)=T\varphi(s_A(\varphi^{-1}(x)))$. 
It is easy to see that $s^{-1}(M)=q$ and $I(s)=I(s_A)=\alpha$.
\end{proof}

An immediate consequence of the above lemma 
is the following result.

\begin{cor}\label{cor:index}
Let $l\in\Z\setminus\{0\}$. For any set of distinct points
$Q=\{q_1, q_2, \ldots, q_{\abs{l}}\} \subset M$ 
and any set of disjoint disc neighbourhoods of these points
there is a local cross section $s$ defined on the union 
of these neighbourhoods such that $s^{-1}(M)=Q$ and $I(s)=l$.
\end{cor}

\end{document}